\documentclass{article}

\usepackage{amsmath,amsthm,amsfonts,amssymb}
\usepackage{tikz}
\usepackage{color}
\usepackage{graphicx}
\usepackage{subcaption} 
\usepackage{comment}
\usepackage{enumerate}
\usepackage{float}

\DeclareMathOperator{\Aut}{Aut}

\DeclareMathOperator{\OG}{\mathcal{OG}}

\DeclareMathOperator{\B}{\mathcal{B}}

\theoremstyle{plain}
\newtheorem{Theorem}{Theorem}

\newtheorem{Proposition}[Theorem]{Proposition}

\newtheorem{Definition}{Definition}
\newtheorem{Remark}{Remark}

\newtheorem{Lemma}[Theorem]{Lemma}

\title{Basic Tetravalent Oriented Graphs of Independent-Cycle Type}
\author{Nemanja Poznanovi\'{c}\footnote{University of Melbourne, Australia.  email: {nempoznanovic@gmail.com}} \hspace{1cm} Cheryl E. Praeger\footnote{University of Western Australia, Australia. email:  {cheryl.praeger@uwa.edu.au}}}

\begin{document}
	\maketitle
	
	\begin{abstract}
		The family $\OG(4)$ consisting of graph-group pairs $(\Gamma, G)$, where $\Gamma$ is a finite, connected, 4-valent graph admitting a $G$-vertex-, and $G$-edge-transitive, but not $G$-arc-transitive action, has recently been examined using a normal quotient methodology.  A subfamily of $\OG(4)$ has been identified as `basic', due to the fact that all members of $\OG(4)$ are normal covers of at least one basic pair. 
		We provide an explicit classification of those basic pairs $(\Gamma, G)$ which have at least two independent cyclic $G$-normal quotients (these are $G$-normal quotients which are not extendable to a common cyclic normal quotient). 
	\end{abstract}
	
	\section{Introduction}
	Finite tetravalent graphs admitting a half-arc-transitive group action have been consistently and actively studied since the 1990s. Several different trends have emerged in this research, often focusing on a particular aspect of these objects, such as their alternating cycles \cite{marusic1998half, marusic1999alt, potocnik2017radius}, their normal quotients \cite{al2017cycle, al2015finite, poznanovic2021four}, their vertex stabilisers \cite{marusic2001point, marusic2005quartic, potovcnik2011vstabunbdd, xia2019tetravalent}, or their role as medial graphs for regular maps on surfaces \cite{marusic1998maps}. 
	In \cite{al2015finite}, a general framework  was introduced for studying the family $\OG(4)$ of graph-group pairs $(\Gamma,G)$, with $\Gamma$ a finite connected tetravalent graph, and $G$ a vertex- and edge-transitive, but not arc-transitive, group of automorphisms.  This method (described in detail below) uses a normal quotient reduction, and has already been successfully used to study other families of graphs exhibiting particular symmetry conditions, see for instance \cite{morris2009strongly, praeger1993nan, praeger1999finite}.  The basic aim of the method is to describe  $\OG(4)$ using graph quotients arising from normal subgroups of the groups contained in the family.

	Given a pair $(\Gamma, G)\in \OG(4)$, and a normal subgroup $N$ of $G$, we may define a \textit{$G$-normal quotient} graph $\Gamma_N$ of $\Gamma$. The vertices of $\Gamma_N$ are the $N$-orbits in $V\Gamma$, and there is an edge between two vertices of $\Gamma_N$ if and only if there is at least one edge between vertices from the corresponding $N$-orbits in $V\Gamma$. The group $G$ then induces a group $G_N$ of automorphisms of $\Gamma_N$, so that we obtain another graph-group pair $(\Gamma_N, G_N)$. The  important result \cite[Theorem 1.1]{al2015finite} then tells us that taking a normal quotient of a pair $(\Gamma, G)\in \OG(4)$ either produces another pair $(\Gamma_N, G_N) \in \OG(4)$ (and in this case $\Gamma$ is a $G$-normal cover of $\Gamma_N$), or the quotient graph $\Gamma_N$ is isomorphic to one of $K_1, K_2$ or $\mathbf{C}_r$ for some $r\geq 3$. In the latter case we say that the quotient is \textit{degenerate}.
	
	In light of this result, we say that a pair $(\Gamma, G) \in \OG(4)$ is \textit{basic} if all of its $G$-normal quotients relative to non-trivial normal subgroups of $G$ are degenerate. By \cite[Theorem 1.1]{al2015finite}, it follows that every member of $\OG(4)$ is a normal cover of at least one basic pair. Hence the authors of \cite{al2015finite} suggest that we may obtain a description of $\OG(4)$ by producing a good description of the basic pairs,  and combining it with a theory to describe their $G$-normal covers.
	
	With this aim in mind, the authors of \cite{al2015finite} further divide  the basic pairs in $\OG(4)$ into three types. A pair $(\Gamma, G) \in \OG(4)$ is said to be basic of \textit{quasiprimitive type} if the only degenerate $G$-normal quotient of $\Gamma$ is $K_1$. This occurs precisely when all non-trivial normal subgroups of $G$ are transitive on $V\Gamma$; such a permutation group is called \emph{quasiprimitive}.
	If the only degenerate $G$-normal quotients of a basic pair $(\Gamma, G) \in \OG(4)$ are the graphs $K_1$ or $K_2$, and $\Gamma$ has at least one $G$-normal quotient isomorphic to $K_2$, then $(\Gamma, G)$ is said to be basic of \textit{biquasiprimitive type}. (The group $G$ here is biquasiprimitive: it is not quasiprimitive but each nontrivial normal subgroup has at most two orbits.) 
	The other basic pairs in $\OG(4)$ must have at least one normal quotient isomorphic to a cycle graph $\mathbf{C}_r$ for some $r\geq 3$, and these basic pairs are said to be of \textit{cycle type.}
	
	In the first paper on this topic \cite{al2015finite}, the authors manage to provide a general description of the basic pairs of quasiprimitive type \cite[Theorem 1.3]{al2015finite} by applying the structure theorem for finite quasiprimitive permutation groups given in  \cite{praeger1993nan}. 
	The basic pairs of biquasiprimitive type were analysed and described in \cite{poznanovic2021four}, where a method analogous to the quasiprimitive case was used, this time applying the less-detailed structure theorem for biquasiprimitive permutation groups in \cite{praeger2003finite}.  A description of the basic biquasiprimitive pairs is given in \cite[Theorem 1.1]{poznanovic2021four}.
	
	The remaining basic pairs are those having at least one cyclic normal quotient. As noted above, these pairs are said to be basic of cycle type. Since there is no general theory describing the groups appearing in these pairs, they are the most difficult to analyse and describe. 
	These pairs are also more complex than both the quasiprimitive and biquasiprimitive basic pairs, in that there is much greater diversity in the possible degenerate quotients they can have. In particular, a basic pair $(\Gamma, G)\in\OG(4)$ of cycle type may have many non-isomorphic cyclic normal quotients. These quotients, despite all being cycles, may have different orders and may be $G$-oriented or $G$-unoriented depending on the $G$-action on the quotient graph, see \cite[Theorem 1]{al2017cycle}. 
	
	In trying to understand the possible types of degenerate quotients which may occur for basic pairs of cycle type, the importance of considering so-called `independent' cyclic normal quotients was demonstrated in \cite{al2017cycle}. Two cyclic normal quotients of a pair ($\Gamma, G) \in \OG(4)$ are said to be \emph{independent cyclic normal quotients} if they are not extendable to a common cyclic normal quotient of $(\Gamma, G)$. 
	
	If $(\Gamma, G)$ is a basic pair of cycle type and \textit{does not} have independent cyclic normal quotients, then all of its cyclic normal quotients will be either $G$-oriented, or all of its cyclic normal quotients will be $G$-unoriented (see \cite[Section 3]{cyclepaper}). Hence the basic pairs of cycle type were further divided into three types in \cite{cyclepaper} as follows. A basic pair $(\Gamma, G) \in \OG(4)$ with cyclic normal quotients is 
	\begin{enumerate}
		\item basic of \textit{oriented-cycle type} if $(\Gamma, G)$ does not have independent cyclic normal quotients \textbf{and} all of its cyclic normal quotients are $G$-oriented,
		\item basic of \textit{unoriented-cycle type} if $(\Gamma, G)$ does not have independent cyclic normal quotients \textbf{and} all of its cyclic normal quotients are $G$-unoriented,
		\item basic of \textit{independent-cycle type} if $(\Gamma, G)$ has independent cyclic normal quotients.	
	\end{enumerate}
	Through a careful analysis, a description of the basic pairs of oriented-cycle type and unoriented-cycle type was obtained in \cite[Theorem 1.1]{cyclepaper}.
	
	On the other hand, the existence of independent cyclic quotients of a pair $(\Gamma, G) \in \OG(4)$ proves to be restrictive enough to allow for a description of all such pairs. In particular, by \cite[Theorem 2]{al2017cycle}, if a pair $(\Gamma, G)$  has independent cyclic normal quotients then it is a $G$-normal cover of a pair $(\overline{\Gamma}, \overline{G})\in \OG(4)$ which belongs to one of six infinite families appearing in \cite[Table 1]{al2017cycle}. The (underlying unoriented)  graphs in these families are all either direct products of two cycles $\mathbf{C}_r \times \mathbf{C}_s$ for some $r,s\geq 3$, or are induced subgraphs or standard double covers of these direct product graphs.
	
	It follows that all basic pairs $(\Gamma, G)$ of independent-cycle type belong to one of the families of examples in \cite[Table 1]{al2017cycle}. The objective of this paper is to identify which members of these families are in fact basic. Our main result is Theorem \ref{IndMainTheorem}. It identifies explicitly these basic members.
	The graphs and groups mentioned in this theorem are all defined in detail in Definitions \ref{IndDef} and \ref{IndDefDouble} in the next section. 
	
	\begin{Theorem}\label{IndMainTheorem} 
		A pair $(\Gamma, G) \in \OG(4)$ is basic of  independent-cycle type if and only if it appears in one of the lines of Table \ref{IndMainTable}. In each row $p$ and $q$ denote arbitrary odd primes.  
		
		\begin{table}[H]
			\centering
			\begin{tabular} {cccl} 
				\hline $\Gamma $& $G$& $(r, s)$ & Defined in  \\
				\hline $\Gamma(r, s) $& $G(r, s)$  & $(4,p), (p,4), $ or $(p,q)$ & Definition \ref{IndDef}\\
				$	\Gamma^{+}(r, s)$ & $G^{+}(r, s)$    & $(4,4), (4,2p)$ or $ (2p,4)$& Definition \ref{IndDef}\\
				$\Gamma(r, s) $&$ H(r, s)$ &  $(p,4)$ or $(p,2q)$& Definition \ref{IndDef} \\
				
				$\Gamma^{+}(r, s) $ & $H^{+}(r, s)$ & $(4,4), (4,2p)$, $ (2p,4)$ or $(2p,2q)$& Definition \ref{IndDef}\\
				$\Gamma_{2}(r, s) $ & $G_{2}(r, s) $&  $(p,q)$& Definition \ref{IndDefDouble} \\
				\hline
			\end{tabular}\caption{The basic pairs of independent-cycle type. }\label{IndMainTable}
		\end{table}
	\end{Theorem}

	Theorem \ref{IndMainTheorem}, when combined with \cite[Theorem 1.1]{cyclepaper}, provides a description of all the basic pairs of cycle type. 
		We note in particular that our classification of the basic pairs  of independent-cycle type is completely explicit. For all other types, apart from one explicit family of examples in~\cite[Theorem 1.1.1(a)]{cyclepaper}, the analysis shows that, for a basic pair $(\Gamma,G)$, the group $G$ has a unique minimal normal subgroup $N$, and tight upper bounds are obtained on the number of simple direct factors of $N$, \cite{al2015finite,poznanovic2019biquasiprimitive, cyclepaper}.  
	
	In the next section we set up our approach to proving Theorem \ref{IndMainTheorem} by outlining the necessary background theory. We then prove this theorem in two parts in Sections \ref{secIndNec} and \ref{secIndSuf}.
	
	\section{Preliminaries}
	All graphs in this paper are simple and finite. Given a graph $\Gamma$ we will let $V \Gamma$ and $E \Gamma$ denote the sets of vertices and edges of $\Gamma$ respectively. Given a vertex $\alpha \in V \Gamma$ we will let $\Gamma(\alpha)$ denote the set of neighbours of $\alpha$. An \textit{arc} of a graph $\Gamma$ is an ordered pair of adjacent vertices. We will let $A \Gamma$ denote the set of arcs of $\Gamma$. Given an arc $(\alpha, \beta) \in A \Gamma$, we will call $(\beta, \alpha)$ its \emph{reverse arc}. In particular, each edge $\{\alpha, \beta\} \in E \Gamma$ has two arcs associated with it, namely $(\alpha, \beta)$ and $(\beta, \alpha)$.
	For other basic graph-theoretic concepts, please refer to \cite{godsil2013algebraic}.
	
	\subsection{G-Oriented Graphs}
	For basic concepts of permutation group theory we refer the reader to \cite{praeger2018permutation}. Given a graph $\Gamma$ and a group $G \leq \Aut(G)$, we say that $\Gamma$ is \textit{$G$-vertex-transitive}, \textit{$G$-edge-transitive}, or \textit{$G$-arc-transitive} if $G$ is transitive on $V\Gamma, E\Gamma$, or $A\Gamma$ respectively.
	
	We will say that a graph $\Gamma$ is \textit{$G$-oriented}, with respect to a group $G\leq \Aut(\Gamma)$, if $G$ is transitive on $V\Gamma$ and $E\Gamma$ but is not transitive on $A\Gamma$.
	In the literature, such graphs are also called \textit{$G$-half-arc-transitive}, or are said to admit \textit{a half-transitive $G$-action}.
	
	Since $G$-oriented graphs are vertex-transitive they are necessarily regular and all of their connected components are isomorphic. Moreover, the valency of a $G$-oriented graph is necessarily even \cite{tutte1966connectivity}. It is easy to see that the connected 2-valent $G$-oriented graphs are oriented cycles. Hence the smallest nontrivial valency a $G$-oriented graph can have is 4. 
	
	We will let $\OG(4)$ denote the family of graph group pairs $(\Gamma, G)$ where $\Gamma$ is a connected 4-valent $G$-oriented graph. 
	For a detailed overview of the family $\OG(4)$, and of $G$-oriented graphs in general, please see \cite{al2015finite} or \cite[Section 2]{cyclepaper}.

	\subsection{Normal Quotients}
	Suppose that $(\Gamma, G) \in \OG(4)$ and let $N$ be a nontrivial normal subgroup of $G$.
	Let $\B_N$ be the partition of $V\Gamma$ into the $N$-orbits, that is, $\B_N = \{(\alpha^N)^g : g \in G\}$ where $\alpha$ is an arbitrary fixed vertex of $\Gamma$, and $\alpha^N$ denotes the $N$-orbit containing $\alpha$. In this case, $\B_N$ is a $G$-invariant partition of $V\Gamma$, so we may define the \textit{$G$-normal-quotient graph} of $\Gamma$ with respect to $N$, denoted $\Gamma_N$. The vertex set of $\Gamma_N$ is the set $\B_N$ of $N$-orbits, and there is an edge between two
	vertices in $\Gamma_N$ if and only if there is at least one edge of $\Gamma$ between vertices from the corresponding $N$-orbits.

	The group $G$ then induces a group $G_N$ of automorphisms of $\Gamma_N$.
	Specifically, $G_N = G/K$, where $K$ is the kernel of the $G$-action on $\Gamma_N$. By definition, $N\leq K$ and since $K$ fixes all $N$-orbits setwise, it follows that the $K$-orbits are the same as the $N$-orbits, so $\Gamma_K=\Gamma_N$. However, $K$ may be strictly larger than $N$.  
	
	It was shown in \cite[Theorem 1.1]{al2015finite} 
	that for any $(\Gamma, G) \in \mathcal{O} \mathcal{G}(4)$, and any nontrivial normal subgroup $N$ of $G$, either $\left(\Gamma_{N}, G_{N}\right)$ is also in $\mathcal{O} \mathcal{G}(4)$, or $\Gamma_{N}$ is isomorphic to $K_{1}$, $K_{2}$ or a cycle $\mathbf{C}_r$, for some $r \geq 3$. 
	Note that if $(\Gamma_N, G_N)$ is itself a member of $\OG(4)$, that is, if $\Gamma_N$ is a 4-valent $G_N$-oriented graph, then $\Gamma$ is said to be a \textit{$G$-normal cover of $\Gamma_N$}.
	
	Following the authors of \cite{al2015finite}, we  say that a pair $(\Gamma_N, G_N)$ is \textit{degenerate} if $\Gamma_N$ is isomorphic to one of $K_1$, $K_2$ or $\mathbf{C}_r$ for some $r \geq 3$. We also define a pair ($\Gamma, G) \in \OG(4)$ to be \textit{basic} if $(\Gamma_N, G_N)$ is degenerate for every nontrivial normal subgroup $N$ of $G$. 
	As mentioned in the introduction, the basic pairs can be divided into three types (quasiprimitive, biquasiprimitive and cycle type) depending on their possible degenerate normal quotient graphs, and those of the first two types have been studied in \cite{al2015finite} and \cite{poznanovic2021four} respectively. Here, we will only be concerned with the basic pairs of cycle type. 
	
	Recall that a pair $(\Gamma, G)\in \OG(4)$ is basic of cycle type if it has at least one $G$-normal quotient which is isomorphic to a cycle graph $\mathbf{C}_r$ for some $r \geq 3$. Suppose that $(\Gamma, G)$ is such a pair with cyclic normal quotient $\Gamma_N \cong \mathbf{C}_r$, for some normal subgroup $N$ of $G$ and $r\geq 3$.
	In such a case, we will always let $\widetilde{N}$ denote the largest normal subgroup of $G$ (containing $N$) which fixes each $N$-orbit setwise, that is, $\widetilde{N}$ will always refer to the kernel of the $G$-action on the set of $N$-orbits in $V\Gamma$. Since $N \leq \widetilde{N}$ and these groups have the same orbits in $V\Gamma$, it will always be the case that $\Gamma_N = \Gamma_{\widetilde{N}}$.
	
	Since $\Gamma_N$ is a cyclic normal quotient of $\Gamma$, it follows 
		from the fundamental analysis in \cite[Theorem 1.1]{al2015finite} that either $\Gamma_N$ is an oriented cycle and the induced group $\widetilde{G} := G/\widetilde{N}$ is its full automorphism group $C_r$, or  $\Gamma_N$ is an unoriented cycle and the induced group $\widetilde{G} := G/\widetilde{N}$ acts arc-transitively, and again is the full automorphism group $D_r$, where $D_r$ denotes the dihedral group of order $2r$. 
	If the normal quotient $(\Gamma_N, G/\widetilde{N}$) of a pair $(\Gamma, G) \in \OG(4)$ is isomorphic to $(\mathbf{C}_r, C_r)$ for some $r\geq 3$, then we will say that the quotient is \textit{$G$-oriented} (or simply that it is oriented). If $(\Gamma_N, G/\widetilde{N}$) is isomorphic to $(\mathbf{C}_r, D_r)$ then we will say that the quotient is \textit{$G$-unoriented} (or simply that it is unoriented).
	
	Of course, if $\Gamma_N$ is a $G$-oriented cyclic normal quotient, then the stabiliser $G_\alpha$ of any vertex will fix the $N$-orbit containing $\alpha$, and so $G_\alpha \leq \widetilde{N}$. On the other hand, if $\Gamma_N$ is $G$-unoriented then any vertex $\alpha$ will have one out-neighbour (and one in-neighbour) in each of the two adjacent $N$-orbits, hence any non-identity automorphism which fixes $\alpha$ must swap at least two $N$-orbits, in particular $G_\alpha \cap \widetilde{N} = 1$. 
	
	Two cyclic normal quotients $\Gamma_M$ and $\Gamma_N$ of $(\Gamma, G) \in \mathcal{O G}(4)$ are \textit{independent} if the normal quotient $\Gamma_K$, where $K=\widetilde{M} \cap \widetilde{N}$, is not a cycle. Thus if $(\Gamma, G) \in \mathcal{O} \mathcal{G}(4)$ is basic of cycle type and $\Gamma_M$ and $\Gamma_N$ are independent cyclic normal quotients of $\Gamma$, then $K=\widetilde{N} \cap \widetilde{M}=1$.
	
	If a basic pair does not have independent cyclic normal quotients then all of its normal quotients must be oriented or all of its normal quotients must be unoriented (see the discussion in the introduction), and such basic pairs have been analysed in \cite{cyclepaper}. The purpose of this paper is to identify those basic pairs $(\Gamma, G) \in \OG(4)$ of independent-cycle type, that is those basic pairs which have independent cyclic normal quotients. 
	
	\subsection{Independent Cyclic Normal Quotients} 
	A description of the pairs $(\Gamma, G) \in \OG(4)$ having independent cyclic quotients was provided in \cite{al2017cycle} via the following theorem.
	
	\begin{Theorem} \label{IndTypesThm} \emph{\cite[Theorem 2]{al2017cycle}}
		Let $(\Gamma, G) \in \mathcal{O G}(4)$, let $\alpha$ be a vertex, and suppose that $(\Gamma, G)$ has independent cyclic normal quotients $\Gamma_{N} \cong \mathbf{C}_{r}$ and $\Gamma_{M} \cong \mathbf{C}_{s}$, where $r \geq 3, s \geq 3$. Then $G_{\alpha} \cong C_{2}$, and the following hold:
		\begin{enumerate}[(a)]
			\item  at least one of $\Gamma_{N}, \Gamma_{M}$ is $G$-unoriented, say $\Gamma_{N}$ is G-unoriented;
			\item $(\Gamma, G) \in \mathcal{O G}(4)$ is a normal cover of $(\overline{\Gamma}, \overline{G}) \in \mathcal{O G}(4)$, which has independent cyclic normal quotients $\overline{\Gamma}_{\overline{N}} \cong \mathbf{C}_{r}$ and $\overline{\Gamma}_{\overline{M}} \cong \mathbf{C}_{s}$ such that $\overline{N} \cap \overline{M}=1$;
			\item$\overline{\Gamma}, \overline{G}$ are as in one of the lines of Table $\ref{IndTypesThmTable}$, and $\overline{\Gamma}_{\overline{M}}$ (and $\Gamma_{M}$) are $G$-oriented if and only if the entry in column $3$ is `Yes'.
		\end{enumerate}
		\begin{table}[H] 
			\centering
			
			\begin{tabular} {llcl} 
				\hline $\overline{\Gamma}$ & $\overline{G} $& Is $\overline{\Gamma}_{\overline{M}}$ \, $G$-oriented?& Conditions on $r, s$  \\
				\hline $\Gamma(r, s) $& $G(r, s)$ & Yes & At least one odd \\
				$	\Gamma^{+}(r, s)$ & $G^{+}(r, s)$ &  Yes & Both even  \\
				$\Gamma(r, s) $&$ H(r, s)$ &  No  & $r $ odd,  $s$ even  \\
				$\Gamma(s, r) $& $H(s, r) $& No  & $r$ even, $s$ odd  \\
				$\Gamma^{+}(r, s) $ & $H^{+}(r, s)$ & No & Both even \\
				$\Gamma_{2}(r, s) $ & $G_{2}(r, s) $&  No  &  Both odd \\
				\hline
			\end{tabular}\caption{Table for Theorem \ref{IndTypesThm}.} \label{IndTypesThmTable}
		\end{table}
	\end{Theorem} 
	
	Before progressing further, we will give the definitions of the graphs and groups appearing in Table $\ref{IndTypesThmTable}$ (and also in Table~\ref{IndMainTable} for Theorem \ref{IndMainTheorem}) as they are stated in \cite{al2017cycle}.
	
	\begin{Definition}\label{IndDef} \emph{\cite[Definition 2.1]{al2017cycle}}
		Let $r, s$ be integers, each at least $3$. Define the undirected graph $\Gamma(r, s)$ to have vertex set $X:=\mathbb{Z}_{r} \times \mathbb{Z}_{s}$, such that a vertex $(i, j) \in X$ is joined by an edge to each of the four vertices $(i \pm 1, j \pm 1) .$ Also, if $r, s$ are both even define
		\begin{equation}\label{X+}
			X^{+}:=\{(i, j) \in X \mid i, j \text { of the same parity }\}
		\end{equation}
		and let $\Gamma^{+}(r, s)=\left[X^{+}\right]$, the induced subgraph.  Define the following permutations of $X$, for $(i, j) \in X$,
		$$
		\begin{array}{ll}
			\mu:(i, j) \mapsto(i+1, j), & \nu:(i, j) \mapsto(i, j+1), \\
			\sigma:(i, j) \mapsto(-i, j), & \tau:(i, j) \mapsto(-i,-j),
		\end{array}
		$$
		and define the groups as in Table $\ref{IndDefTable}$, where in lines $3$ and $4$ ($r, s$ both even), we identify $\mu, \nu, \sigma, \tau$ with their restrictions to $X^{+}$, and consider the subgroups $G^{+}(r, s)$ and $H^{+}(r, s)$ acting on $X^{+}$.
		
		\begin{table}[H]
			\centering
			\begin{tabular}{llll}
				
				\hline $\Gamma$ & $G$ & \text {Generators for } $G$ & Conditions on $r, s $\\
				\hline $\Gamma(r, s)$ & $G(r, s)$ & $\mu, \nu, \sigma $& - \\
				$\Gamma(r, s) $&  $H(r, s)$ & $\mu, \sigma \nu, \tau$ & $s \hbox{ even} $ \\
				$\Gamma^{+}(r, s)$ & $G^{+}(r, s)$ & $\mu^{2}, \mu \nu, \sigma $&  Both even \\
				$	\Gamma^{+}(r, s) $& $H^{+}(r, s) $&$ \mu^{2}, \sigma \mu \nu, \tau $& Both even  \\
				\hline
				
			\end{tabular}	\caption{Table for Definition \ref{IndDef}.}\label{IndDefTable}
		\end{table}
		Also let 
		\begin{enumerate}
			\item[(a)]  $\widetilde{M}=\langle\mu, \sigma\rangle \cong D_{r}, $ $ M=\langle\mu\rangle \cong C_{r},$ $ N=\langle \nu\rangle \cong C_{s}$, and $N^{\#}=\left\langle \nu^{2}, \tau \sigma \nu\right\rangle$, and note that $G(r, s)=\widetilde{M} \times N$ and $H(r, s)=\left(M \times N^{\#}\right) \cdot\langle\tau\rangle$ (with $s$ even; also note that $(\sigma \nu)^{2}=\nu^{2})$;
			\item[(b)]  $M_{t}=\left\langle\mu^{t}\right\rangle$ for $t \mid r$, and $N_{t}=\left\langle\nu^{t}\right\rangle$ for $t \mid s$. If $r$ and $s$ are both even, we also consider the following subgroups restricted to their actions on $X^+$: 
			\[
			\widetilde{M}^{+}=\left\langle\mu^{2}, \sigma\right\rangle \cong D_{r},\  M^{+}=M_{2} \cong C_{r / 2},\ \mbox{and}\ N^{+}=N_{2} \cong C_{s / 2}.
			\]
		\end{enumerate}

	\end{Definition} 
	
	We will need to consider the standard double covers (sometimes called canonical double covers)  of the graphs $\Gamma(r,s)$ defined in Definition \ref{IndDef}.
	The \textit{standard double cover} of a graph $\Gamma$ with vertex set $X$ is the graph $\Gamma_{2}$ with vertex set $X_{2}=\left\{x_{\delta} \mid x \in X, \delta \in \mathbb{Z}_{2}\right\}$ such that $\left\{x_{\delta}, y_{\delta^{\prime}}\right\}$ is an edge if and only if $\delta \neq \delta^{\prime}$ and $\{x, y\}$ is an edge of $\Gamma$.
	Note that $\Gamma_{2}$ has the same valency as $\Gamma$ and twice the number of vertices.
	
	\begin{Definition}\label{IndDefDouble}
		\emph{\cite[Construction 2.10]{al2017cycle}} Let $r, s$ be positive integers, with $r, s \geq 3$, and let $\Gamma_{2}(r, s)$ be the standard double cover of the graph $\Gamma(r, s)$ of Definition $\ref{IndDef}$, so $\Gamma_{2}(r, s)$ has vertex set $X_{2}=\left\{x_{\delta} \mid x \in X, \delta \in \mathbb{Z}_{2}\right\}$ where $X=\mathbb{Z}_{r} \times \mathbb{Z}_{s}$, the vertex set of $\Gamma(r,s)$. Note that a vertex $(i,j)_\delta \in X_2$ is adjacent to each of the four vertices $(i\pm1, j \pm1)_{\delta^{'}}$ where $\delta \neq \delta^{'}$.
		
		
		We extend the automorphisms defined in Definition $\ref{IndDef}$ to maps on $X_{2}$ as follows. For $(i, j)_{\delta} \in X_{2}$
		$$
		\begin{array}{ll}
			\mu:(i, j)_{\delta} \mapsto(i+1, j)_{\delta}, & \nu:(i, j)_{\delta} \mapsto(i, j+1)_{\delta}, \\
			\sigma:(i, j)_{\delta} \mapsto(i,-j)_{\delta+1}, & \tau:(i, j)_{\delta} \mapsto(-i,-j)_{\delta}
		\end{array}
		$$
		and let $G_{2}(r, s)=\langle\mu, \nu, \sigma, \tau\rangle=\widehat{M} \times \widehat{N}$, where $\widehat{M}=\langle\mu, \sigma \tau\rangle \cong D_{r}$, and $\widehat{N}=$ $\langle \nu, \sigma\rangle \cong D_{s}$.
		
	\end{Definition}
	
	By Theorem \ref{IndTypesThm}, if a pair $(\Gamma, G) \in \OG(4)$ has independent cyclic normal quotients, then $(\Gamma, G)$ is a normal cover of a pair $(\overline{\Gamma}, \overline{G}) \in \mathcal{O G}(4)$  appearing in Table \ref{IndTypesThmTable}, and so every basic pair of independent-cycle type appears in Table \ref{IndTypesThmTable}, though of course not all such pairs are basic. We now make some preliminary remarks about our proof strategy. 
	
	\begin{Remark}\label{rem1}
		{\rm
			Since every basic pair of independent-cycle type appears in Table \ref{IndTypesThmTable}, all we need to do in order to determine these basic pairs (and hence prove Theorem \ref{IndMainTheorem}), is to decide which of the pairs in Table \ref{IndTypesThmTable} are basic. 
			For the purpose of doing this, we in fact only need to determine which of the pairs appearing in Rows 1, 2, 3, 5, 6 of Table \ref{IndTypesThmTable} are basic. The reason for this is the following. 
			
			If $(\Gamma, G) \in \OG(4)$ is basic of independent-cycle type having independent $G$-unoriented cyclic normal quotients $\Gamma_{N} \cong \mathbf{C}_{r}$ and $\Gamma_{M} \cong \mathbf{C}_{s}$, with $r \geq 3$ and $ s \geq 3$ having different parities, then by Theorem \ref{IndTypesThm},  $(\Gamma, G)$ is as described in Rows 3 or 4 of Table \ref{IndTypesThmTable}. Since both quotients are $G$-unoriented,  we could interchange $N$ and $M$, and $r$ and $s$, if necessary, and assume that $r$ is odd and hence that $(\Gamma, G)$ appears in Row 3 of Table \ref{IndTypesThmTable}. 
			
			Hence without loss of generality,  all basic pairs $(\Gamma, G)\in \OG(4)$ of independent-cycle type belong to one of the five families given in Table \ref{IndTypesTable} below.  Note that the five rows of this table correspond exactly to the five rows of Table \ref{IndMainTable} (and of Table \ref{IndTypesThmTable} apart from row 4). The elements in the column `Generators' of  Table \ref{IndTypesTable} form a generating set for $G$ as given in Table~\ref{IndDefTable} and Definition~\ref{IndDefDouble}.
			The subgroups listed in the column labelled `Normal Subgroups' in  Table \ref{IndTypesTable} are notable normal subgroups of the group $G$ which we will use in later arguments. Note also that for all groups $G$ listed in Table \ref{IndTypesTable}, all subgroups of the groups in the `Normal Subgroups' column are also normal in $G$. This is easily checked by noting the following relations between the  various generating elements of these groups: 
			$$\mu\nu = \nu\mu, \hspace{15mm}
			\mu^\sigma = \mu^{\mp1}, \hspace {5mm}  \nu^{\sigma} = \nu^{\pm1},  \hspace{15mm}
			\mu^\tau = \mu^{-1}, \hspace {5mm}  \nu^{\tau} = \nu^{-1},$$
			where the result of conjugating by $\sigma$ depends on  whether or not $G = G_2(r,s)$.     
		}
	\end{Remark}

	\begin{table}[H] 
		\centering
		\begin{tabular} {llccc} 
			\hline ${\Gamma}$ & ${G} $& Conditions on $r, s$  &Generators & Normal Subgroups\\
			\hline $\Gamma(r, s) $& $G(r, s)$  & At least one odd &$\mu, \nu, \sigma $ & $\langle \mu \rangle, \langle \nu \rangle$\\
			$	\Gamma^{+}(r, s)$ & $G^{+}(r, s)$ & Both even  &$\mu^{2}, \mu \nu, \sigma $& $\langle \mu^2 \rangle, \langle \nu^2 \rangle$\\
			$\Gamma(r, s) $&$ H(r, s)$ &   $r $ odd,  $s$ even  &$\mu, \sigma \nu, \tau$& $\langle \mu \rangle, \langle \nu^2 \rangle$\\
			$\Gamma^{+}(r, s) $ & $H^{+}(r, s)$  & Both even &$ \mu^{2}, \sigma \mu \nu, \tau $ &$\langle \mu^2 \rangle, \langle \nu^2 \rangle$\\
			$\Gamma_{2}(r, s) $ & $G_{2}(r, s) $ &  Both odd &$\mu, \nu, \sigma, \tau $&$\langle \mu \rangle, \langle \nu \rangle$\\
			\hline
		\end{tabular}\caption{The five possible kinds of basic pairs of independent-cycle type.} \label{IndTypesTable}
	\end{table}
	
	We may now prove Theorem \ref{IndMainTheorem} by determining for each row of Table \ref{IndTypesTable}, the values of $(r,s)$ which produce a basic pair $(\Gamma, G)\in \OG(4)$.
	More precisely, we prove Theorem \ref{IndMainTheorem} over the next two sections as follows. First, we show that if $(\Gamma, G) \in \OG(4)$ appears in Table \ref{IndTypesTable} and is basic, then $(r,s)$ has one of the values appearing in the corresponding row of Table \ref{IndMainTable}. This therefore is a necessary condition and we prove it in Section \ref{secIndNec}. Then in Section \ref{secIndSuf}, we show that for each of the values of $(r,s)$ appearing in Table \ref{IndMainTable}, the pair $(\Gamma, G)$ is in fact basic.
	
	\section{Proof of Theorem \ref{IndMainTheorem}: Necessary Condition}\label{secIndNec}
	
		For the pairs $(\Gamma, G) \in \OG(4)$ with independent cyclic normal quotients, that is, those appearing in Table \ref{IndTypesTable}, we begin by restricting the possible values of $(r,s)$ which can occur when the pair is basic.  
	We obtain precisely the values of $(r,s)$ given in Theorem \ref{IndMainTheorem}. The following two lemmas identify these values for $r$ (Lemma \ref{IndCyclicSubgroupsMu}), and $s$ (Lemma \ref{IndCyclicSubgroupsNu}).

	\begin{Lemma}\label{IndCyclicSubgroupsMu}
		Suppose that $(\Gamma, G)\in \OG(4)$ is as described in one of the rows of Table $\ref{IndTypesTable}$  with appropriate $r,s \geq 3$. Suppose further that $(\Gamma, G)$ is basic of cycle type. 
		Then the following hold:
		\begin{enumerate}[(a)]
			\item 	If $(\Gamma, G)$ is as in Row $1, 3$, or $5$ of Table $\ref{IndTypesTable}$ then $r \in \{4,p\}$ for some odd prime $p$.
			\item If $(\Gamma, G)$ is as in Row $2$ or $4$ of Table $\ref{IndTypesTable}$ then $r \in \{4,2p\}$ for some odd prime $p$.
		\end{enumerate}
	\end{Lemma}
	\begin{proof}
		(a) \quad Suppose that $(\Gamma, G)$ is as in Rows 1,  3, or 5. Then  $M:= \langle \mu \rangle\cong C_r$ is a normal subgroup of $G$, and hence also all subgroups of $M$ are normal in $G$. If $r$ is a prime then $r$ is odd, since $r\geq3$, and  is one of the possibilities in part (a). Suppose now that $r$ is not a prime. Then $M$ contains a proper nontrivial subgroup $L$. Since $L$ is a normal subgroup of $G$ contained in $M$, and since $M$ has at least three orbits on $\Gamma$, it follows that $\Gamma_L$ is a non-trivial proper $G$-normal quotient, and hence is a cycle, since $(\Gamma,G)$ is basic.
		
		Next, notice that the vertex $(0,0)$ is adjacent to $(-1,-1), (1,-1), (-1,1) $ and $(1,1)$ (or if $\Gamma$ is as in Row 5, then $(0,0)_0$ is adjacent to $(-1,-1)_1,$ $ (1,-1)_1, (-1,1)_1 $ and $(1,1)_1$). Thus exactly two of these neighbours lie in one $L$-orbit and the other two lie in another $L$-orbit. Since $\mu : (i,j) \mapsto (i+1,j)$  (respectively, $\mu:(i, j)_{\delta} \mapsto(i+1, j)_{\delta}$) has no effect on the second coordinate of each vertex, it follows that $(-1,1) $ and $(1,1)$ (respectively $(-1,1)_1 $ and $(1,1)_1$) lie in the same $L$-orbit. 
		
		
		Since $M$ is semiregular on $V\Gamma$, this implies that $\mu^2$, and hence $\langle \mu^2 \rangle$, lies in $L$. Then since $L \neq M$, it follows that $L = \langle \mu^2 \rangle$. Hence $\mu$ has even order and so $r= 2k$ for some integer $k>1$, and $ \langle \mu^2 \rangle$ has order $k$. Now consider the subgroup $K:= \langle \mu^k\rangle$ of $M$ of order 2. By the same reasoning  $\Gamma_K$ is a cycle and $K$ must contain $\mu^2$, implying that $k = 2$ and $r =4$. This proves part (a).
		
		(b)\quad Suppose now that $(\Gamma, G)$ is as in Row 2 or 4. Then $r$ is even and $M_2 := \langle \mu^2 \rangle$ is normal in $G$, as are each of its subgroups. Also, since $r\geq3$, the subgroup $M_2\ne 1$. 
		Suppose that $M_2$ is not simple and consider a proper nontrivial subgroup $L$ of $M_2$. By the same argument as in case (a), $\Gamma_L$ is a cycle and the vertices $(-1,1)$ and $(1,1)$ must lie in the same $L$-orbit. Again,  since $M_2$ is semiregular on vertices, it follows that $\mu^2 \in L$. This is a contradiction, so $M_2 \cong C_{r/2}$ has prime order, meaning that  $r =2p$ where $p$ is a prime. Thus part (b) is proved.
	\end{proof}
	
	\begin{Lemma}\label{IndCyclicSubgroupsNu}
		Suppose that $(\Gamma, G)\in \OG(4)$ is as described in one of the rows of Table $\ref{IndTypesTable}$  with appropriate $r,s \geq 3$. Suppose further that $(\Gamma, G)$ is basic of cycle type. 
		Then the following hold:
		\begin{enumerate}[(a)]
			\item 	If $(\Gamma, G)$ is as in Rows $1$ or $5$ of Table $\ref{IndTypesTable}$ then $s \in \{4,q\}$ for some odd prime $q$.
			\item If $(\Gamma, G)$ is as in Rows $2, 3$, or $4$ of Table $\ref{IndTypesTable}$ then $s \in \{4,2q\}$ for some odd prime $q$.
		\end{enumerate}
	\end{Lemma}
	\begin{proof}
		Note that in all cases $\langle \nu \rangle$ and $\langle \nu^2 \rangle$ are semiregular on vertices. The proofs of (a) and (b) here can thus be  handled using virtually identical arguments to the proofs of (a) and (b) of Lemma \ref{IndCyclicSubgroupsMu} by swapping the roles of $\mu$ and $\nu$, and noting that $\nu$ only affects the second coordinate of a vertex (instead of the first).
	\end{proof}
	
	Before giving our main result for this section we need one more result for the special case where $G$ has a normal subgroup of order $2$.
	
	\begin{Lemma} \label{Indlength2orbit}
		Suppose that $(\Gamma, G)\in \OG(4)$ is as described in one of Rows $1-4$ of Table $\ref{IndTypesTable}$ with appropriate $r,s \geq 3$. If $(\Gamma, G)$ is basic of cycle type with $L$ a normal subgroup of $G$ of order $2$, then one of $r, s$ is equal to $4$. 
		Moreover, in this case $L$ must swap $(0,0)$ with one of $(2,0), (2,2),$ or $(0,2)$.
	\end{Lemma}
	\begin{proof}
		Suppose that $(\Gamma, G)$ is basic of cycle type. Then since $|\Gamma| >4$ and $|L|=2$, the quotient $\Gamma_L$ has at least three vertices, and hence $\Gamma_L$ must be a cycle.  Consider the $L$-orbit containing the vertex $(0,0)$. This $L$-orbit contains exactly one other vertex, say $\alpha$. The neighbours of $(0,0)$ in $\Gamma$ are precisely $(-1,-1), (1,-1), (-1,1) $ and $(1,1)$, and since $\Gamma_L$ is a cycle it follows that these four vertices must all be adjacent to $\alpha$. 
		
		By Definition \ref{IndDef} we know that the vertex set of $\Gamma$ is either $X$ or $X^+$, and we may check from Definition~\ref{IndDef} that in either case, the neighbourhoods  of $(-1,-1), (1,-1), (-1,1) $ and $(1,1)$ are as follows: 
		\begin{align*}
			\Gamma((-1,-1)) &= \{(0,0), (-2,0), (-2,-2), (0,-2)\}, \\
			\Gamma((1,-1))  &= \{(0,0), (2,0), (2,-2), (0,-2)\}, \\
			\Gamma((-1,1))  &= \{(0,0), (-2,0), (-2,2), (0,2)\}, \\
			\Gamma((1,1))   &= \{(0,0), (2,0), (2,2), (0,2)\}. 
		\end{align*}
		The four vertices $(-1,-1), (1,-1), (-1,1), (1,1)$ therefore have a common neighbour other than $(0,0)$ if and only if at least one of $2 \equiv -2 $ (mod $r$), or $2 \equiv -2 $ (mod $s$). Hence either $r$ or $s$ divides 4, and so (since $r, s\geq 3$) one of $r,s$ must be equal to 4. Thus $\alpha$ must be one of $(2,0), (2,2),$ or $(0,2)$, and as $L$ swaps $(0,0)$ and $\alpha$, the second part follows. 
	\end{proof}

	Finally, we put together these results to prove Theorem \ref{1to5Necc} which gives the possible values of $(r,s)$ for Theorem \ref{IndMainTheorem}.
	
	\begin{Theorem} \label{1to5Necc}
		Suppose that $(\Gamma, G)\in \OG(4)$ is as described in one of the rows of Table $\ref{IndTypesTable}$  with appropriate $r,s \geq 3$. Suppose that $(\Gamma, G)$ is basic of cycle type. Then $\Gamma, G, (r,s)$ are as in the appropriate Row of Table~$\ref{1to5NeccTable}$.
	\end{Theorem}
	
	\begin{table}[H] 
		\centering
		\begin{tabular} {cllll} 
			\hline 
			Row of Table~\ref{IndTypesTable} & ${\Gamma}$ & ${G} $& Possibilities for $(r, s)$  &Conditions \\
			\hline $1$ & $\Gamma(r, s) $& $G(r, s)$  & $(4,p), (p,4), $ or $(p,q)$ &$p$ and $q$ are odd primes\\
			$2$& $\Gamma^{+}(r, s)$ & $G^{+}(r, s)$ & $(4,4), (4,2p)$ or $ (2p,4)$ &$p$ is an odd prime\\
			$3$& $\Gamma(r, s) $&$ H(r, s)$ &  $(p,4)$ or $(p,2q)$  &$p$ and $q$ are odd primes\\
			$4$&$\Gamma^{+}(r, s) $ & $H^{+}(r, s)$  & $(4,4), (4,2p)$, $ (2p,4)$ or $(2p,2q)$ &$p$ and $q$ are odd primes\\
			$5$& $\Gamma_{2}(r, s) $ & $G_{2}(r, s) $ &  $(p,q)$ &$p$ and $q$ are odd primes\\
			\hline
		\end{tabular}\caption{Possibilities for $(r,s)$ for basic pairs of independent-cycle type.} \label{1to5NeccTable}
	\end{table}

	\begin{proof}
		Combining Lemmas \ref{IndCyclicSubgroupsMu} and \ref{IndCyclicSubgroupsNu} we obtain the following possibilities for $r$ and $s$ in each of the Rows of Table~\ref{IndTypesTable}.
		\begin{align*}
			&\hbox{In Row 1, } &r \in \{4,p\} &\hbox{ and } s \in \{4,q\},\\
			&\hbox{in Row 2, } &r \in \{4,2p\}  &\hbox{ and } s \in \{4,2q\},\\
			&\hbox{in Row 3, } &r \in \{4,p\}   &\hbox{ and } s \in \{4,2q\},\\
			&\hbox{in Row 4, } &r \in \{4,2p\}&\hbox{ and } s \in \{4,2q\},\\
			&\hbox{in Row 5, } &r \in \{4,p\}   &\hbox{ and } s \in \{4,q\},
		\end{align*}
		where $p$ and $q$ are odd primes (possibly equal). Hence there are at most four possibilities in each case. We will now show that some of these possibilities cannot occur. For most arguments, we will refer to the conditions on $r,s$ in Table \ref{IndTypesTable}.
		
		In Row 1, $(r,s) \neq (4,4)$, since one of $r,s$ must be odd. Similarly, in Row 3, $(r,s) \neq (4,4), (4,2q)$ since $r$ is odd, and in Row 5 only  $(r,s) = (p,q)$ is possible since $r$ and $s$ are both odd. In Row 4, we do not exclude any possibility.
		Finally, consider Row 2 with $(r,s) = (2p,2q)$ where $p$ and $q$ are odd primes. Here we have  $p = q + 2\ell$ for some $\ell$. Then since $G^+(r,s) = \langle \mu^{2}, \mu \nu, \sigma\rangle$ and $\mu$ and $\nu$ commute, we get
		$(\mu\nu)^q\cdot (\mu^2)^\ell= \mu^p\nu^q \in G^+(r,s)$, and $\langle \mu^p\nu^q\rangle$ is a normal subgroup of $G^+(r,s)$ of order 2 (using $\mu^{2p}=\nu^{2q}=1$ and the definition of $\sigma$ from Definition~\ref{IndDef}). This however contradicts  Lemma \ref{Indlength2orbit} since neither of $r,s$ is equal to 4. 
		This completes the proof.
	\end{proof}
	
	\section{Proof of Theorem \ref{IndMainTheorem}: Sufficient Condition}\label{secIndSuf}
	In this section we complete the proof of Theorem \ref{IndMainTheorem} by showing that for each of the values of $(r,s)$ given in Table \ref{IndMainTable}, the appropriate pair $(\Gamma, G)$ from the same table will be basic of cycle type. 
	We do this by showing that, for all such $(\Gamma, G)$ and $(r,s)$, all of the normal quotients of the graph $\Gamma$ with respect to minimal normal subgroups of the group $G$ are cycles. 
	
	Since all groups $G$ appearing in Table \ref{IndMainTable} are isomorphic to direct products of cyclic or dihedral groups (or quotients of such groups), we begin with the following lemma concerning minimal normal subgroups of such groups. 
	\begin{Lemma} \label{IndMinDrCr}
		Suppose $G= H \times K$  where $H \cong D_r$, and $K \cong C_s$ or $K \cong D_s$ for some $r,s \geq 3$.  \\
		If $M$ is a minimal normal subgroup of $G$ then either
		\begin{enumerate}[(a)]
			\item $M$ is a minimal normal subgroup of $H$, so $M \cong C_p$ where $p$ is a prime divisor of $r$; or
			\item $M$ is a minimal normal subgroup of $K$, so $M \cong C_p$ where $p$ is a prime divisor of $s$; or
			\item both $r$ and $s$ are even, $M \leq Z(G)$, and $M \cong C_2$.
		\end{enumerate}  
	\end{Lemma}
	\begin{proof}
		Let $M$ be a minimal normal subgroup of $G$. If $M \cap H \neq 1$ then $M \cap H$ is a normal subgroup of $G$ and hence must be equal to $M$ by assumption. In particular, $M$ is a minimal normal subgroup of $G$ contained in $H$.  If $R\neq 1$ is a normal subgroup of $H$ contained in $M$, then it is also a normal subgroup of $G$, so $R= M$. Hence $M$ must be a minimal normal subgroup of $H \cong D_r$ and so $M \cong C_p$ where $p$ is a prime divisor of $r$.  If $M \cap K \neq 1$ then the same argument shows that $M$ is a minimal normal subgroup of $K$, so $M \cong C_p$ where $p$ is a prime divisor of $s$.
		
		Suppose on the other hand that $1 = M \cap H = M\cap K$. Then $M$ projects nontrivially onto each of $H$ and $K$ and we claim that $M \leq Z(G)$. To see this, take $m \in M$ where $m= h\cdot k$ where $h \in H$ and $k \in K$. Now take an arbitrary $g \in H$ and note that since $M$ is normal in $G$, it follows that $m^{-1}\cdot m^g \in M$. Now since $k$ commutes with all elements in $H$, we get $m^{-1}\cdot m^g = k^{-1}h^{-1}g^{-1}hkg = h^{-1}g^{-1}hg = [h,g] \in M$. Thus $[h,g] \in M \cap H = 1$, and since $g$ was arbitrary it follows that $h$ commutes with every element of $H$. An analogous argument shows that $k$ commutes with each element of $K$. Hence $h \in Z(H)$ and $k \in Z(K)$ and so $m  = hk \in Z(H)\times Z(K) = Z(G)$. Therefore $M \leq Z(G)$.  
		
		The conditions $M\cap H=M\cap K=1$ imply that the projection maps from $H\times K$ to $H$, and from $H\times K$ to $K$, restrict to isomorphisms from $M$ to a subgroup of $Z(H)$, and of $Z(K)$, respectively. In particular both $Z(H)$ and $Z(K)$ are nontrivial. 	
		Now $Z(D_r) = 1$ if $r$ is odd, and $Z(D_r) \cong C_2$ if $r$ is even. Thus $r$ is even, and $M\cong Z(H)\cong C_2$. Similarly, if $K=D_{s}$ then $s$ is even. On the other hand, if $K \cong C_s$ then, since $Z(K)$ contains a subgroup isomorphic to $M\cong  C_2$, $s$ is even in this case also.
	\end{proof}

We now go through each of the rows of Table \ref{IndTypesTable} and show that, for each of the values of $(r,s)$  in Table~\ref{1to5NeccTable} (which were deduced in Theorem \ref{1to5Necc}), the corresponding pair $(\Gamma, G)$ appearing in Table \ref{IndTypesTable} is basic of cycle type. The next result deals with the first row of Table \ref{IndTypesTable}. 

\begin{Proposition}
	Suppose that $\Gamma:= \Gamma(r,s)$ and $G:= G(r,s)$ are as described in Row 1 of  Table \ref{IndTypesTable}.  If $(r,s)$ is of the form $(4,p), (p,4), $ or $(p,q)$  where $p$ and $q$ are odd primes, then $(\Gamma, G)$ is basic of independent-cycle type.
\end{Proposition}
\begin{proof}
	By Definition~\ref{IndDef}(a), $G = \widetilde{M} \times N = \langle \mu, \sigma\rangle \times \langle \nu \rangle \cong  D_r \times C_s$, and  by \cite[Lemma 2.7]{al2017cycle}, $\Gamma_N$ is a $G$-unoriented cycle and  $\Gamma_M$ is a  $G$-oriented cycle, and these are independent cyclic normal quotients. Thus we need only check that $(\Gamma(r,s), G(r,s))$  is basic of cycle type. Since one of $r,s$ is odd,  Lemma \ref{IndMinDrCr} tells us that a minimal normal subgroup of $G$ is either minimal normal in $\widetilde{M}$ or in $N$. 
	The $N$-orbits on $V\Gamma = \mathbb{Z}_r \times \mathbb{Z}_s$ are of the form $ B_{i}=\left\{(i, j) \mid j \in \mathbb{Z}_{s}\right\}$, for $ i \in \mathbb{Z}_{r}$, while the $\widetilde{M}$-orbits (and $M$-orbits where $M = \langle \mu \rangle $) are of the form $E_{j}=\left\{(i, j) \mid i \in \mathbb{Z}_{r}\right\}$, for $ j \in \mathbb{Z}_{s}$. 
	
	It is easy to check that if $(r,s) = (p,q)$, where $p$ and $q$ are odd primes then $(\Gamma(p,q), G(p,q))$  is basic of cycle type: by Lemma \ref{IndMinDrCr}, $G$ has two minimal normal subgroups $M = \langle \mu \rangle$ and $N = \langle \nu \rangle$, and $\Gamma_M$ and $\Gamma_N$ are both cycles. 
	
	On the other hand, if $\Gamma = \Gamma(4,p)$, and $G = G(4,p)$ where $p$ is an odd prime, then again $G$ will have two minimal normal subgroups. These are $N \cong C_p$ and $L \leq M$ where 
	$L := \langle \mu^2\rangle \cong C_2$. We know that $\Gamma_N$ is a cycle so we only need to check that $\Gamma_L$ is a cycle. To see that it is, we use the fact that the $L$-orbits are contained in the $M$-orbits which are  $\{E_{j}:  j \in \mathbb{Z}_{p}\}$. One can check that the $L$-orbits are of the form $E_{0,j} = \{(0,j),(2,j)\}$ or $E_{1, j} = \{(1,j),(3,j)\}$ for $j \in \mathbb{Z}_p$. In particular, we may check that in $\Gamma_L$,  the $L$-orbit  $E_{0,0}$ has only one out-neighbour $E_{1,1}$ and one in-neighbour $E_{1,-1}$. Hence $\Gamma_L$ is a cycle and $(\Gamma, G)$ is basic of cycle type.
	
	
	If $\Gamma = \Gamma(p,4)$ then $G$ again has two minimal normal subgroups by Lemma \ref{IndMinDrCr}, namely $M= \langle \mu \rangle$ and $J = \langle \nu^2 \rangle$. Here $\Gamma_M$ is a cycle and one can check that $\Gamma_J$ is a cycle by noting that the $J$-orbits are contained in the $N$-orbits which are $\{B_i : i \in \mathbb{Z}_p\}$. The $J$-orbits are $B_{i,0} = \{(i,0),(i,2)\}$ and $B_{i,1} = \{(i,1),(i,3)\}$ for $i \in \mathbb{Z}_p$, and $B_{0,0}$ has just two neighbours in $\Gamma_J$, namely $B_{1,1}$ and $B_{-1,1}$. Therefore $\Gamma_J$ is a cycle and hence  $(\Gamma, G)$ is basic of cycle type.
\end{proof}

To deal with the second row of Table \ref{IndTypesTable} we need the following result about minimal normal subgroups.

\begin{Lemma} \label{Index2}
	Suppose that $H$ is an index $2$ subgroup of a group $G$. Let $M$ be a minimal normal subgroup of $H$. Then either 
	\begin{enumerate}[(i)]
		\item $M$ is a minimal normal subgroup of $G$ contained in $H$; or 
		\item $M\times M^g$ is a normal subgroup of $G$ contained in $H$, where $g\in G\backslash H$.
	\end{enumerate}
\end{Lemma}
\begin{proof}
	If $M$ is normal in $G$ then $M$ must be minimal normal in $G$ and (i) holds. Otherwise, if $M$ is not normal in $G$, then there is an element $g \in G \backslash H$ such that $M^g \neq M$. Now $G = \langle H, g \rangle$, and it follows that $M^g$ is a minimal normal subgroup of $H^g = H$ and so $M^* := \langle M, M^g\rangle = M \times M^g$ is a normal subgroup of $G$ contained in $H$. 
\end{proof}

\begin{Proposition}
	Suppose that $\Gamma^+:= \Gamma^+(r,s)$ and $G^+:= G^+(r,s)$ are as described in Row 2 of Table \ref{IndTypesTable}.   If $(r,s) $ is of the form $(4,4), (4,2p)$ or $ (2p,4)$ where $p$ is an odd prime then $(\Gamma^+, G^+)$ is basic of independent-cycle type.
\end{Proposition}

\begin{proof}
	By Definition~\ref{IndDef},  $G^+ = \langle \mu^2, \mu \nu, \sigma\rangle$ is an index 2  subgroup of $G=G(r,s)$.  Neither $\mu$ nor $\nu$ is contained in $G^+$ (since $G \neq G^+$, see Table \ref{IndDefTable}), while both $\mu^2$ and $\mu^{-2}(\mu\nu)^2 = \nu^2$ are contained in $G^+$. It  follows that for any element of the form $\mu^i\sigma^\epsilon\nu^j$ in $G^+$,  with $\epsilon \in \{0,1\}$,  $i$ and $j$ must have the same parity. To see this, suppose that this were not the case and that $\mu^i\sigma^\epsilon\nu^j\in G^+$ with $i$ even and $j$ odd. Then, noting that both $r$ and $s$ are even, appropriate repeated left multiplication by $\mu^2, \sigma$, and $\nu^{-2}$ in order would give $\nu \in G^+$, a contradiction. (If $i$ is odd and $j$ is even, an analogous argument gives the contradiction $\mu\in G^+$.) Thus $N^+ =\langle \nu^2\rangle$ and $M^+=\langle \mu^2\rangle$ are normal subgroups of $G^+$, and by \cite[Lemma 2.7(b)]{al2017cycle}, $\Gamma^+_{N^+}=C_r$ is a $G^+$-unoriented cycle and  $\Gamma^+_{M^+}=C_s$ is a  $G^+$-oriented cycle, and these are independent cyclic normal quotients, so we need only check that $(\Gamma^+, G^+)$  is basic of cycle type.
	
	First we consider the case $(\Gamma^+, G^+) = (\Gamma^+(4,4), G^+(4,4))$. It is easy to check (for instance using  \textsc{Magma}\cite{bosma1997magma}) that $G^+(4,4)$ has three minimal normal subgroups each having order 2. These are $M^+, N^+$ and $\langle \mu^2\nu^2\rangle$. It is also easy to check that $\Gamma^+(4,4) \cong K_{4,4}$. Letting $A = \{(0,0), (2,0), (2,2), (0,2)\}$ and $B = V\Gamma^+ \backslash A$ denote the bipartition of $\Gamma^+$, it is clear that each of these three minimal normal subgroups swaps the vertex $(0,0)$ with some vertex in $A$. In all three cases the normal quotient with respect to this subgroup is a 4-cycle so $(\Gamma^+, G^+)$ is basic of cycle type.

	Now suppose that $(r,s)$ is $(4,2p)$ or $ (2p,4)$, where $p$ is an odd prime.  Since $G^+$ is an index 2 subgroup of $G$, each minimal normal subgroup of $G^+$ satisfies case (i) or (ii) of Lemma \ref{Index2}. The minimal normal subgroups of $G^+$ which are minimal normal in $G$ are precisely $M^+=\langle \mu^2\rangle$ and  $N^+=\langle \nu^2 \rangle$. This can be checked using Lemma \ref{IndMinDrCr} and noting the fact that the central minimal normal subgroups of $G(2p, 4)$ and  $G(4,2p)$ are $\langle\mu^p\nu^2\rangle$ and $\langle\mu^2\nu^p\rangle$ respectively, and neither of these is contained in $G^+$ for the reason given in the first paragraph of the proof. Each of 
	$\Gamma^+_{M^+}$ and $\Gamma^+_{N^+}$ is a cyclic normal quotient of $\Gamma^+$. 
	
	We now show that $M^+$ and  $N^+$ are the only minimal normal subgroups of $G^+$, by showing that $G^+$ has no minimal normal subgroups of the form described in case (ii) of Lemma \ref{Index2}.
	To this end, suppose that $L$ is such a minimal normal subgroup of $G^+$, and that $L \times L^g$ is normal in $G$ (and contained in $G^+$) for some $g \in G\backslash G^+$. Since $G^+$ has order $8p$ this is only possible if $|L| = 2$. 
	
	Consider an element  $\gamma \in G^+$ of order 2. Then viewing $G$ as the direct product $G = \langle \mu, \sigma\rangle \times \langle \nu \rangle$, we may write $\gamma = (x,y)$ where $x $ is equal to $1, \mu^{r/2}, $ or $\mu^i\sigma$ for some $i$, and $y$ is equal to $1$ or $\nu^{s/2}$.  Suppose that $L = \langle \gamma \rangle $. Since $L$ is normal in $G^+$, we have $\gamma^{\mu\nu} = \gamma$, and hence $x^{\mu} =x$. Now $(\mu^i\sigma)^{\mu} = \mu^{i-2}\sigma\ne \mu^i\sigma$ since $|\mu|=r>2$, and it follows that $x \in \{1, \mu^{r/2}\}$. 
	Moreover, since $L\neq 1$ and $L$  is not a minimal normal subgroup of $G$, it follows that $\gamma = \mu^{r/2}\nu^{s/2}\in G^+$ (dropping the direct product notation).  Since one of $r, s$ equals $4$, it follows that one of $\mu^{r/2},\nu^{s/2}$, say $z$, lies in $G^+$, and hence also $z\gamma\in G^+$. Hence both of $\mu^{r/2},\nu^{s/2}$ lie in $G^+$, and this is a contradiction since one of $r, s$ equals $2p$ with $p$ an odd prime.  We conclude  that $M^+$ and  $N^+$ are the only minimal normal subgroups of $G^+$, and hence that $(\Gamma^+,G^+)$ is basic of cycle type.
\end{proof}

In order to produce similar results for rows 3 and 4 of Table \ref{IndTypesTable}, we need some preliminary theory on the groups $H(r,s)$. The following lemma describes minimal normal subgroups of the group $H(r,s)$ for certain parameters $(r,s)$.

\begin{Lemma} \label{HplusLemma}
	Let $p$ and $q$ be  odd primes, and let $H(r,s) = \langle \mu, \nu^2 ,\tau\sigma\nu, \tau \rangle$ as in Table~\ref{IndTypesTable}.
	\begin{enumerate}[(a)] 
		\item If $(r,s) = (p,4)$ then the minimal normal subgroups of $H(r,s)$ are $\langle \mu \rangle$ and $\langle \nu^2 \rangle$.
		\item If $(r,s) = (p,2q)$ then the minimal normal subgroups of $H(r,s)$ are $\langle \mu \rangle$ and $\langle \nu^2 \rangle$.
		\item If $(r,s) = (2p,2q)$ then the minimal normal subgroups of $H(r,s)$ are $\langle \mu^2\rangle $, $\langle \mu^p \rangle $ and $\langle \nu^2 \rangle$.  
		\item If $(r,s) = (2p,4)$ then the minimal normal subgroups of $H(r,s)$ are $\langle \mu^2\rangle$, $\langle \mu^p \rangle$, $\langle \nu^2 \rangle$ and $\langle \mu^p\nu^2 \rangle$.
		\item If $(r,s) = (4,2p)$ then the minimal normal subgroups of $H(r,s)$ are $\langle \mu^2\rangle $ and $\langle \nu^2 \rangle$.

	\end{enumerate}
\end{Lemma}
\begin{proof}
	By Definition \ref{IndDef}(a), we may view $H: = H(r,s)$ as $H = \langle \mu \rangle \times \langle \nu^2, \tau\sigma\nu\rangle \rtimes \langle \tau \rangle \cong (C_r \times D_{s/2}) \rtimes C_2$. Note that the order of $H$ is $2rs$. In cases $(a)$ and $ (b)$, $\langle \mu \rangle$ and $\langle \nu^2 \rangle$ are minimal normal subgroups of $H$, while for each of the remaining cases, $\langle \mu^2 \rangle$, $\langle \mu^{r/2}\rangle $ and $\langle \nu^2 \rangle$ are minimal normal subgroups of $H$ (of course, $\langle \mu^2 \rangle$ and  $\langle \mu^{r/2}\rangle $ are equal if $r = 4$).
	Finally, if $(r,s) = (2p,4)$ then $\langle \mu^p\nu^2\rangle$ is also a minimal normal subgroup of $H$ since it is generated by a central involution.
	
	We will now show that in each of these cases, the minimal normal subgroups of $H$ mentioned above are the only minimal normal subgroups of $H$. In cases $(a) - (d)$, if we compute the centraliser of $M$ in $H$,  where $M := \langle \mu \rangle$ in cases $(a)$ and $(b)$, and $M := \langle \mu^2 \rangle$ in cases $(c)$ and $(d)$, we get $C_H(M) =  \langle \mu, \nu^2, \tau\sigma\nu\rangle $. Thus in cases  $(a) - (d)$, since distinct minimal normal subgroups centralise each other, all minimal normal subgroups of $H(r,s)$ are contained in $H_0:= \langle \mu, \nu^2, \tau\sigma\nu\rangle \cong C_r \times D_{s/2}$, and we note that $|H_0| = rs$ and $H=\langle H_0, \tau\rangle$. 
	
	Suppose now that $L$ is a minimal normal subgroup of $H$ which is not minimal normal in $H_0$. In this case, $L$ contains a  proper subgroup $K$ such that $K$ is minimal normal in $H_0$. Then $K^\tau$ is also a minimal normal subgroup of $H_0^\tau = H_0$ and so $\langle K, K^\tau \rangle = K \times K^\tau$ is a normal subgroup of $H$. Since $K \leq L$ and $K^\tau \leq L^\tau = L$, we have  $K \times K^\tau \leq L$ and $K \times K^\tau$ is normalised by $\langle H_0, \tau\rangle =H$. Hence $K \times K^\tau = L$, so every minimal normal subgroup of $H$ is either minimal normal in $H_0$ or is the direct product of two isomorphic minimal normal subgroups of $H_0$ which are interchanged by $\tau$. It thus suffices to check all minimal normal subgroups of $H_0$ and determine which of these give rise to  minimal normal subgroups of $H$ in this way.

	If $(r,s) = (p,4)$, then $H_0 \cong C_p \times C_2^2$ and we may check that the minimal normal subgroups of $H_0$ are $\langle \mu \rangle$, $\langle \nu^2 \rangle$, $K_1:= \langle \tau\sigma\nu \rangle$ and $K_2:= \langle \tau\sigma\nu^{-1} \rangle$. Note that conjugation by $\tau$ interchanges the subgroups $K_1$ and $K_2$; however $K_1\times K_2$ contains $\langle \nu^2 \rangle$ and so $K_1\times K_2$ is not a minimal normal subgroup of $H$. Thus $\langle \mu \rangle$ and $\langle \nu^2 \rangle$ are the only minimal normal subgroups of $H$. 
	
	If $(r,s) = (p,2q)$, then $H_0 \cong C_p \times D_{q}$ and,  by Lemma \ref{IndMinDrCr}, the minimal normal subgroups of $H_0$ are $\langle \mu \rangle $ and $\langle \nu^2 \rangle$ and both are minimal normal in $H$.
	If $(r,s) = (2p,2q)$, then $H_0 \cong C_{2p} \times D_{q}$ and,  by Lemma \ref{IndMinDrCr}, the minimal normal subgroups of $H_0$ are  $\langle \mu^2 \rangle$, $\langle \mu^p \rangle$ and $\langle \nu^2 \rangle$ and each of these is minimal normal in $H$.
	
	If $(r,s) = (2p,4)$, then $H_0 \cong C_{2p}  \times D_2 \cong C_p \times C_2^3$ is abelian with order $8p$. Thus $\langle \mu^2 \rangle$ is the unique subgroup of of $H_0$ of order $p$. The other minimal normal subgroups of $H_0$ are all generated by involutions and hence are $\langle \mu^p \rangle$, $\langle \nu^2 \rangle$ and $\langle \mu^p\nu^2 \rangle$, as well as $K_1:= \langle \tau\sigma\nu \rangle$,  $K_2:= \langle \tau\sigma\nu^{-1} \rangle$, $J_1:= \langle \mu^p\tau\sigma\nu \rangle$ and $J_2:= \langle \mu^p\tau\sigma\nu^{-1} \rangle$. Now conjugation by $\tau$ in $H$ fixes $\langle \mu^2 \rangle$, $\langle \mu^p \rangle$, $\langle \nu^2 \rangle$ and $\langle \mu^p\nu^2 \rangle$, and  swaps $K_1$ with $K_2$, and $J_1$ with $J_2$. Moreover, both $K_1 \times K_2$ and $J_1 \times J_2$ contain the subgroup $\langle \nu^2 \rangle$ and therefore are not minimal normal in $H$. Thus the only minimal normal subgroups of $H$ are $\langle \mu^2 \rangle$, $\langle \mu^p \rangle$, $\langle \nu^2 \rangle$ and $\langle \mu^p\nu^2 \rangle$.
	Hence we have proved the result for cases $(a) - (d)$.
	
	Finally, in case $(e)$, we have $(r,s) = (4,2p)$ with $|H| = 2rs = 16p$. Now letting $H_1 := C_H(\nu^2)$, we see that $H_1 \geq \langle \mu, \nu^2, \sigma \nu^{-1} \rangle$. In particular since $\nu^2$ is not central in $H$, we have $16p = |H| > |H_1| = |\langle \mu, \nu^2, \sigma \nu^{-1} \rangle| > |\langle \mu, \nu^2 \rangle| = |C_4 \times C_p| =  4p$. This implies that $|H_1| = 8p$, $H_1$ is an index 2 subgroup of $H$, and $H_1 = \langle \mu, \nu^2, \sigma \nu^{-1} \rangle$. In particular, all minimal normal subgroups of $H$ are contained in $H_1$.
	
	Now notice that since $p$ is odd,   we have $\nu^{p+1} \in H_1$,  and so $(\sigma\nu^{-1})\nu^{p+1} = \sigma \nu^p \in H_1$ and this is an involution. Moreover,  $\langle \mu, \sigma \nu^p\rangle \cong D_4$ is normal in $H_1$. In particular $H_1 = \langle \nu^2 \rangle \times \langle \mu, \sigma \nu^p\rangle \cong C_p \times D_4$, and so by Lemma \ref{IndMinDrCr} 
	the minimal normal subgroups of $H_1$ are $\langle \nu^2 \rangle$ and $\langle \mu^2 \rangle$. If $L$ is a minimal normal subgroup of $H$ which is not  minimal normal in $H_1$ then again $L = K \times K^\tau$, for $K,  K^\tau$ isomorphic minimal normal subgroups of $H_1$. In particular, no such subgroup  $L$ exists since both $\langle \nu^2 \rangle$ and $\langle \mu^2 \rangle$ are fixed by conjugation by $\tau$, and so these are the only minimal normal subgroups of $H$.
\end{proof}

	In Propositions~\ref{IndHMinSubs} and~\ref{IndHMinSubs2} we apply the information in Lemma~\ref{HplusLemma} to analyse the pairs $(\Gamma(r,s), H(r,s))$ described in Rows 3 and 4 of Table~\ref{IndTypesTable}. Our proofs rely on \cite[Lemma 2.9]{al2017cycle} which specifies (with $M, N^\#, M^+, N^+$ as in Definition~\ref{IndDef}):
	\begin{itemize}
		\item if $r$ is odd then $\Gamma_M$ and $\Gamma_{N^{\#}}$ are independent $H(r,s)$-unoriented cyclic normal quotients; while
		\item  if $r$ is even then $\Gamma_{M^+}$ and $\Gamma_{N^{+}}$ are independent $H^+(r,s)$-unoriented cyclic normal quotients. 
	\end{itemize}

\begin{Proposition} \label{IndHMinSubs}
	Suppose that $\Gamma:= \Gamma(r,s)$ and $H:= H(r,s) = \langle \mu, \nu^2 ,\tau\sigma\nu, \tau \rangle$  are as described in  Row 3 of Table \ref{IndTypesTable}.  If $(r,s) $ is of the form $ (p,4)$ or $(p,2q)$ where $p$ and $q$ are  odd primes then $(\Gamma, H)$ is basic of independent-cycle type.
\end{Proposition}
\begin{proof}
	
	
	By parts (a) and (b) of Lemma \ref{HplusLemma}, if $(r,s)$ is as described in the assumption, then the  minimal normal subgroups of $H$ are $M:= \langle \mu \rangle$ and $N:= \langle \nu^2 \rangle$. By \cite[Lemma 2.9(a)]{al2017cycle}, $\Gamma_M$ and $\Gamma_{N^{\#}}$ are independent cyclic normal quotients, where $N^{\#}=\left\langle \nu^{2}, \tau \sigma \nu\right\rangle$. Moreover, the $N^\#$-orbits are $B_{i}=\left\{(i, j) \mid j \in \mathbb{Z}_{s}\right\}, \text { for } i \in \mathbb{Z}_{r}$, and the $N$-orbits (which are contained in the $N^\#$-orbits) are $B_{i,0} = \{(i,2k), \mid k \in \mathbb{Z}_s\}$ and $B_{i,1} = \{(i,2k+1), \mid k \in \mathbb{Z}_s\}$ for $i \in \mathbb{Z}_p$. Note that the $N$-orbit  $B_{0,0}$ has just two neighbours in $\Gamma_N$, namely $B_{1,1}$ and $B_{-1,1}$. Therefore $\Gamma_N$ is a cycle and  $(\Gamma, H)$ is basic of independent-cycle type.
\end{proof}

\begin{Proposition}\label{IndHMinSubs2}
	Suppose that $\Gamma^+:= \Gamma^+(r,s)$ and $H^+:= H^+(r,s)$ are as described in Row 4 of Table \ref{IndTypesTable}. If $(r,s) $ is of the form $(4,4), (4,2p),  (2p,4)$ or $(2p,2q)$ where $p$ and $q$ are odd primes then $(\Gamma^+, H^+)$ is basic of independent-cycle type. 
\end{Proposition}

\begin{proof}
	Note that   $H^+ = \langle \mu^{2}, \sigma \mu \nu, \tau \rangle$ is an index 2 subgroup of $H:= H(r,s) = \langle \mu, \sigma \nu, \tau \rangle$, and that $(\sigma\mu\nu)^2 = \nu^2$. In particular,  $H = \langle H^+, \mu \rangle$, and  $H^+$ has order $rs$. Moreover, $M^+:=\langle \mu^2 \rangle, N^+:=\langle \nu^2 \rangle$ are normal subgroups of $H^+$, and by \cite[Lemma 2.9(b)]{al2017cycle}, $\Gamma^+_{M^+}$ and $\Gamma^+_{N^+}$ are independent cyclic normal quotients.  So it is sufficient to prove that $(\Gamma^+, H^+)$ is basic of cycle type. 
	
	Suppose first that $(r,s) = (4,4)$. We check (eg. using  \textsc{Magma} \cite{bosma1997magma}) that $H^+(4,4)$ has three minimal normal subgroups each of order 2. These are $M^+, N^+$ and $\langle \mu^2\nu^2\rangle$. It is also easy to check that $\Gamma^+(4,4) \cong K_{4,4}$. Letting $A = \{(0,0), (2,0), (2,2), (0,2)\}$ and $B = V\Gamma^+ \backslash A$ denote the bipartition of $\Gamma^+$, it is clear that each of these three minimal normal subgroups swaps the vertex $(0,0)$ with some vertex in $A$. Thus in all three cases the normal quotient with respect to this subgroup produces a 4-cycle so $(\Gamma^+, H^+)$ is basic of cycle type.
	
	Now suppose that  $(r,s)$ has one of the three other forms stated in the assumption.
	Let $L$ be a minimal normal subgroup of $H^+$. Since $H^+$ is an index 2 subgroup of $H$ we may again apply Lemma \ref{Index2}.  In particular  either (i) $L$ is a minimal normal subgroup of $H$ contained in $H^+$; or (ii) $L$ is not normal in $H$ and $L\times L^h$ is a normal subgroup of $H$ contained in $H^+$, where $h \in H \backslash H^+$. In each of the  following three cases for $(r,s)$, we determine all possibilities for $L$ and show that for each of these the corresponding normal quotient $\Gamma_L$ is a cycle. 

	First suppose that $(r,s) = (2p,4)$ for some odd prime $p$.  By Lemma \ref{HplusLemma}, the minimal normal subgroups of $H$ are $M^+$, $\langle \mu^p \rangle$, $N^+$ and $\langle \mu^p\nu^2 \rangle$, and of these, only $M^+$ and $N^+$ are contained in $H^+$. Moreover, $M^+$ and $N^+$ are both minimal normal in $H^+$,  and for these two subgroups the corresponding normal quotients are cycles by \cite[Lemma 2.9(b)]{al2017cycle}. Suppose that $H^+$ also has a minimal normal subgroup $L$ of type (ii). Since $|H^+|=8p$, the order $|L| = 2$.  In particular, $L \leq Z(H^+)$, and so $L \leq Z(H^+) \leq C_{H^+}( \langle \mu^2 \rangle)\leq C_{H}( \langle \mu^2 \rangle) = H_0:= \langle  \mu, \nu^2, \tau\sigma\nu \rangle \cong  C_{2p} \times C_2^2\cong C_p\times C_2^3$. It is easy to check that the only involutions in $H_0$ contained in $H^+$ are $\nu^2, \tau\sigma\nu\mu^p$, and $\tau\sigma\nu^3\mu^p$. (To check the latter two elements, note that $ \tau, \sigma\mu\nu, \mu^{p-1}, \nu^2 \in H^+$, and $\tau \cdot (\sigma\mu\nu) \cdot \mu^{p-1} = \tau\sigma\nu\mu^p$, while $\tau \cdot (\sigma\mu\nu) \cdot \mu^{p-1} \cdot \nu^2 = \tau\sigma\nu^3\mu^p$.)
	Since $\tau\sigma\nu\mu^p$ and $\tau\sigma\nu^3\mu^p$ are interchanged under conjugation by $\tau$,  it follows that neither generates a normal subgroup of $H^+$. Moreover, we have already seen that $\nu^2$ generates the normal subgroup $N^+$ of $H^+$ of type (i). Hence there are no minimal normal subgroups of $H^+$ of type (ii), and we conclude that  $(\Gamma^+, H^+)$ is basic of cycle type.

	Next let $(r,s) = (4,2p)$ for some odd prime $p$. By Lemma \ref{HplusLemma}, the minimal normal subgroups of $H$ are $M^+:= \langle \mu^2\rangle $ and $N^+:= \langle \nu^2 \rangle$, these are both minimal normal in $H^+$, and the corresponding normal quotients are cycles by \cite[Lemma 2.9(b)]{al2017cycle}.  Suppose that $H^+$ also has a minimal normal subgroup $L$ of type (ii).  Then $L$ is not normal in $H$, and since $|H^+| = 8p$, it follows as in the previous case that $|L| = 2$ and $L\leq Z(H^+) \leq C_{H^+}(\langle \nu^2\rangle ) \leq C_H(\langle \nu^2\rangle ) = H_1:= \langle \nu^2 \rangle \times \langle \mu, \sigma \nu^p\rangle \cong C_p \times D_4$ (the last part follows from the last paragraph of the proof of Lemma \ref{HplusLemma}). The involutions in $H_1$ are $\mu^2$ and $\sigma\nu^p\mu^i$ for $i \in\{0,1,2,3\}$.  Of these, only $\mu^2$, $\sigma\nu^p\mu$ and $\sigma \nu^p \mu^{3}$ are contained in $H^+$. Since conjugation by $\tau \in H^+$ swaps $\sigma\nu^p\mu$ and $\sigma \nu^p \mu^{3}$, neither of these two involutions generates a normal subgroup of $H^+$, and $\mu^2$ generates the normal subgroup $M^+$ of type (i). Thus there are no minimal normal subgroups of $H^+$ of type (ii), and we conclude that  that  $(\Gamma^+, H^+)$ is basic of cycle type.

	Finally suppose that $(r,s) = (2p,2q)$ for  odd primes $p,q$ so $|H^+| = 4pq$.  By Lemma \ref{HplusLemma}, the minimal normal subgroups of $H$ are  $\langle \mu^2\rangle $, $\langle \mu^p \rangle $ and $\langle \nu^2 \rangle$. Of these $M^+ = \langle \mu^2\rangle$ and $N^+ = \langle \nu^2 \rangle$ are contained in $H^+$; they are minimal normal in $H^+$, and the corresponding normal quotients are cycles by \cite[Lemma 2.9(b)]{al2017cycle}. Suppose that $H^+$ also has a minimal normal subgroup  $L$ of type  (ii), so $L$ is not normal in $H$. Then $|L|^2$ divides $|H^+|$ and so either $|L| = 2$, or $|L| = p$ (if $p=q$). 
	If $|L| = 2$ then $L \leq C_{H^+}(S)$ where $S := \langle \mu^2\rangle \times \langle \nu^2 \rangle$.  Now $|H^+ : S| = 4$ and the set of $S$-cosets in $H^+$ is $ \{S, S\tau, S\sigma\mu\nu, S\sigma\mu\nu\tau\}$. Conjugation by any element of $S\tau \cup S\sigma\mu\nu$ inverts $\mu^2$, while conjugation by any element of  $S\sigma\mu\nu\tau$ inverts $\nu^2$. Hence $C_{H^+}(S) = S$. Thus $L \leq S \cong C_p \times C_q$, which is a contradiction since $S$ contains no involutions. Hence $|L|\ne 2$, and so $p=q$ and $|L|=p$. 
	
	Therefore we have  $|H^+|  =4p^2$ and  $L^* := L\times L^\mu\leq H^+$ with $L^*$ normal in $H$ and $|L^*|=p^2$. It follows that $L^*$ is the unique Sylow $p$-subgroup of $H^+$ and hence $L^*$ contains $\mu^2$ and $\nu^2$ (elements of order $p$). This implies that $L^* = \langle \mu^2 \rangle \times \langle \nu^2 \rangle  \cong C_p \times C_p$. Also $L \neq \langle \mu^2 \rangle , \langle \nu^2 \rangle$ since $L$ is not normal in $H$, and so $L = \langle \mu^2\nu^{2j}\rangle$ for some $j \in \{1,\dots, p-1\}$. Since $L$ is normalised by $\sigma \mu \nu \in H^+$, 
	we have $(\mu^2\nu^{2j})^{\sigma\mu\nu} = \mu^{-2}\nu^{2j} \in L$, and hence $\mu^2\nu^{2j}\cdot \mu^{-2}\nu^{2j} = \nu^{4j} \in L$. Since $|L|=p$ is an odd prime and $\gcd(4j,p)=1$, this implies that $\nu\in L$, which is a contradiction.   Thus there are no minimal normal subgroups of $H^+$ of type (ii), and we conclude that  that  $(\Gamma^+, H^+)$ is basic of cycle type. 
\end{proof}

Finally we deal with Row 5 of Table \ref{IndTypesTable}.

\begin{Proposition}
	Suppose that $\Gamma_2:= \Gamma_2(r,s)$ and $G_2:= G_2(r,s)$ are as described in Row 5 of Table \ref{IndTypesTable}. If $(r,s) =(p,q)$ where $p$ and $q$ are odd primes, then $(\Gamma_2, G_2)$ is basic of independent-cycle type. 
\end{Proposition}

\begin{proof}
	By Definition \ref{IndDefDouble}, $G_{2}(p, q)=\langle\mu, \nu, \sigma, \tau\rangle=\widehat{M} \times \widehat{N}$, where $\widehat{M}=\langle\mu, \sigma \tau\rangle \cong D_{p}$, and $\widehat{N}=$ $\langle \nu, \sigma\rangle \cong D_{q}$. Thus by Lemma \ref{IndMinDrCr}, the minimal normal subgroups of $G_2(p,q)$ are $M:= \langle \mu \rangle \cong C_p$ and $N:= \langle \nu \rangle \cong C_q$.
	By \cite[Lemma 2.11]{al2017cycle}, $({\Gamma_2})_{\widehat{M}}$ and $({\Gamma_2})_{\widehat{N}}$ are independent cyclic normal quotients of $\Gamma_2$. 
	
	Moreover, the $\widehat{N}$-orbits on vertices are $B_i=\left\{(i, j)_\delta \mid j \in \mathbb{Z}_q, \delta \in \mathbb{Z}_2\right\}$ for $i \in \mathbb{Z}_p$, and the $\widehat{M}$-orbits are $C_j=\left\{(i, j)_\delta \mid i \in \mathbb{Z}_p, \delta \in \mathbb{Z}_2\right\}$ for $j \in \mathbb{Z}_q$. The $M$- and $N$-orbits are contained in the $\widehat{M}$- and $\widehat{N}$-orbits respectively, and we check that the $N$-orbits are $B_{i,\epsilon} = \{(i,j)_\epsilon \mid j \in \mathbb{Z}_q$\} for $i \in \mathbb{Z}_p$ and $\epsilon \in \{0,1\}$, and the $M$-orbits are $C_{j,\epsilon} = \{(i,j)_\epsilon \mid i \in \mathbb{Z}_p\}$ for $j \in \mathbb{Z}_q$ and $\epsilon \in \{0,1\}$.
	Now in $(\Gamma_2)_M$, the neighbours of $C_{0,0}$ are $C_{-1,1}$ and $C_{1,1}$, and in $(\Gamma_2)_N$, the neighbours of $B_{0,0}$ are $B_{-1,1}$ and $B_{1,1}$. Hence both $(\Gamma_2)_M$ and $(\Gamma_2)_N$ are cycles and we conclude that $(\Gamma_2, G_2)$ is basic of independent-cycle type.
\end{proof}

To summarise, we have now completed the proof of Theorem~\ref{IndMainTheorem}: if a pair $(\Gamma, G) \in \OG(4)$ is basic of  independent-cycle type, then by Theorem~\ref{1to5Necc}, $\Gamma, G, (r,s)$ satisfy ones of the rows of Table~\ref{IndMainTable}.  Conversely, it follows from the five propositions in this section that each of these entries does indeed yield a pair which is basic of  independent-cycle type.

\section*{Acknowledgements}
The second author acknowledges funding from the ARC Discovery Project Grant DP230101268.
Both authors thank the two anonymous reviewers for some helpful comments which improved the exposition.

\bibliographystyle{plain}
\bibliography{Bibliography}

\end{document}